\documentclass[11pt,twoside, leqno]{article}

\usepackage{amssymb}
\usepackage{amsmath}
\usepackage{mathrsfs}
\usepackage{amsthm}
\usepackage{amsfonts}
\usepackage{enumerate}
\usepackage{esint}
\usepackage{latexsym,amssymb}
\usepackage{color}

\allowdisplaybreaks

\newtheorem{theorem}{Theorem}[section]
\newtheorem{lemma}[theorem]{Lemma}
\newtheorem{corollary}[theorem]{Corollary}
\newtheorem{proposition}[theorem]{Proposition}
\theoremstyle{definition}

\newtheorem{definition}[theorem]{Definition}
\numberwithin{equation}{section}

\numberwithin{equation}{section}

\begin{document}

\arraycolsep=1pt

\title{\bf\Large
Attainability and criticality for multipolar Rellich inequality
\footnotetext {\hspace{-0.35cm}
2020 {\it Mathematics Subject Classification}. 26D10, 42B37, 35A23, 35B25.
\endgraf {\it Key words and phrases}.
Multipolar Rellich inequalities; sharp constant; attainability; criticality.
\endgraf {\it Funding information}: NNSF of China (12071431).
}}
\author{Yongyang Jin, Shoufeng Shen and Li Tang}
\date{}
\maketitle

\vspace{-0.6cm}

\begin{center}
\begin{minipage}{13.5cm}
{\small {\bf Abstract} \quad
In this paper we obtain optimal  multipolar Rellich inequality in $\mathbb{R}^{N}$ for biharmonic Schr\"{o}dinger operator with positive multisingular potentials of the form
\begin{equation*}
H_{n}:=\Delta^{2}-\frac{N^{2}(N-4)^{2}}{n^{4}}\sum_{1\leq i<j\leq n}\frac{\left|a_{i}-a_{j}\right|^{2}}{|x-a_{i}|^{2}|x-a_{j}|^{2}}
     \left(\sum_{1\leq k<l\leq n}\frac{\nu_{k,i,j}\nu_{l,i,j}\left|a_{k}-a_{l}\right|^{2}}{|x-a_{k}|^{2}|x-a_{l}|^{2}}\right),
\end{equation*}
where $a_{1},\cdots,a_{n}$ are $n$ different singular poles, $\nu_{k,i,j}=\frac{N+2n-4}{N}$ when $k=i$ or $k=j$, otherwise $\nu_{k,i,j}=\frac{N-4}{N}$. Moreover, we prove the attainability of the best constant $\frac{N^{2}(N-4)^{2}}{n^{4}}$ in $D^{2,2}(\mathbb{R}^{N})$ for $n\geq3$ and the criticality of the biharmonic Schr\"{o}dinger operator $H_{n}$.

.}
\end{minipage}
\end{center}

%
%
%
%
%

\section{Introduction}
The classical $L^2-$Hardy inequality
\begin{equation}\label{T-1-1}
\int_{\mathbb{R}^{N}}|\nabla u|^{2}dx\geq\frac{(N-2)^{2}}{4}\int_{\mathbb{R}^{N}}\frac{|u|^{2}}{|x|^{2}}dx
\end{equation}
holds for $u\in{C}^{\infty}_{0}(\mathbb{R}^{N})$ when $n\geq 3$. Hardy inequality are of fundamental importance in various branches of mathematics such as partial differential equation, spectral theory, analysis, mathematica physics, geometry and quantum mechanics, and have been comprehensively studied, just to name a few, see for example \cite{BE,BFT,BV,Da,Du,DP,JZ,L,T,VZ,YKG} and references therein.

At the 1954 ICM in Amsterdam \cite{R}, Rellich presented the following inequality as
the second order version of the Hardy inequality for $N\geq5$:
\begin{equation}\label{T-1-2}
\int_{\mathbb{R}^{N}}|\Delta u|^{2}dx\geq\frac{N^{2}(N-4)^{2}}{16}\int_{\mathbb{R}^{N}}\frac{|u|^{2}}{|x|^{2}}dx, \hspace{2mm} \forall u\in C_{0}^{\infty}(\mathbb{R}^{N}).
\end{equation}
The constant $\frac{N^{2}(N-4)^{2}}{16}$ is sharp but can not be attained   by nontrivial functions. Much effort has been
devoted to the study of Rellich type inequalities and their applications.  The reader may refer  \cite{BMO,DH,EE,EL,JS,MOW} and the references therein for further details.
\begin{definition}(\emph{Sharp/Optimal constant})
Denote the norm $\|\cdot\|_{\sigma,2}$, $\sigma=1,2$, as $\|u\|_{1,2}=\|\nabla u\|_{2}$ and $\|u\|_{2,2}=\|\Delta u\|_{2}$. Denote Schr\"{o}dinger/biharmonic Schr\"{o}dinger operator $H=\Delta^{\sigma}-\lambda^{*}V$. We say that the constant $\lambda^{*}$ of the following Hardy/Rellich type inequality associated with $H$
\begin{equation}\label{T-1-3}
\|u\|_{\sigma,2}^{2}\geq\lambda^{*}\int_{\mathbb{R}^{N}}V|u|^{2}dx
\end{equation}
is \emph{sharp} or \emph{optimal} if it is defined through the optimization problem
\begin{equation*}
\lambda^{*}:=\inf_{u\in D^{\sigma,2}(\mathbb{R}^{N})\setminus\{0\}}\frac{\|u\|_{\sigma,2}^{2}}{\int_{\mathbb{R}^{N}}V|u|^{2}dx},
\end{equation*}
where
\begin{equation*}
D^{\sigma,2}(\mathbb{R}^{N}):=\overline{C_{0}^{\infty}(\mathbb{R}^{N})}^{\|\cdot\|_{\sigma,2}}.
\end{equation*}
\end{definition}
It is well known that the  constants $\frac{(N-2)^{2}}{4}$ and $\frac{N^{2}(N-4)^{2}}{16}$ in inequality (\ref{T-1-1}) and (\ref{T-1-2}) are both sharp but not attained.

Recently, motivated by problems in quantum mechanics, many authors  studied multipolar Hardy inequalities and their applications to Schr\"{o}dinger operators with multipolar potentials. The readers may find some relevant results in \cite{BDE,CP,CPT,CZ,FFK,FMT,FT} and the reference therein. The two most studied potentials are
\begin{equation*}
 W_{1}(x)=\sum_{i=1}^{n}\frac{1}{|x-a_{i}|^2}\hspace{2mm}, \hspace{4mm}W_{2}(x)=\sum_{1\leq i<j\leq n}\frac{\left|a_{i}-a_{j}\right|^{2}}{|x-a_{i}|^{2}|x-a_{j}|^{2}}.
\end{equation*}
In order to obtain a lower bound of the spectrum of the Schr\"{o}dinger operators $-\Delta-\mu W_{1}$, $\mu\in(0,\frac{(N-2)^{2}}{4}]$, Bosi-Dolbeault-Esteban in \cite{BDE} proved for $u\in C_{0}^{\infty}(\mathbb{R}^{N})$ the following multipolar Hardy inequality
\begin{equation}\label{T-1-4}
\int_{\mathbb{R}^{N}}|\nabla u|^{2}dx\geq \frac{(N-2)^{2}}{4n}\int_{\mathbb{R}^{N}}W_{1}|u|^{2}dx
    +\frac{(N-2)^{2}}{4n^2}\int_{\mathbb{R}^{N}}W_{2}|u|^{2}dx.
\end{equation}
Several years later, Cazacu-Zuazua obtained in \cite{CZ} the following inequality
\begin{equation}\label{T-1-5}
\int_{\mathbb{R}^{N}}|\nabla u|^{2}dx\geq\frac{(N-2)^{2}}{n^2}\int_{\mathbb{R}^{N}}W_{2}|u|^{2}dx,\hspace{2mm}u\in C_{0}^{\infty}(\mathbb{R}^{N}),
\end{equation}
where the constant $\frac{(N-2)^{2}}{n^2}$ is sharp for $N\geq3$. The inequality (\ref{T-1-5}) is an improvement of (\ref{T-1-4}) in the sense that: the total mass near $a_{i}$ of the potential in (\ref{T-1-5}) is $\frac{(N-2)^{2}}{4}\frac{4n-4}{n^2}\frac{1}{|x-a_{i}|^2}$, which is strictly larger than the total mass $\frac{(N-2)^{2}}{4}\frac{2n-1}{n^2}\frac{1}{|x-a_{i}|^2}$ near $a_{i}$ in inequality (\ref{T-1-4}).  Another multipolar Hardy inequality proved in \cite{DFP} by Devyver-Fraas-Pinchover reads as
\begin{equation}\label{T-1-6}
\int_{\mathbb{R}^{N}}|\nabla u|^{2}dx\geq \left(\frac{N-2}{n+1}\right)^{2}\int_{\mathbb{R}^{N}}\left(W_{1}+W_{2}\right)|u|^{2}dx,\hspace{2mm}u\in C_{0}^{\infty}(\mathbb{R}^{N}),
\end{equation}
where the constant $\left(\frac{N-2}{n+1}\right)^{2}$ is sharp. Furthermore, they gave a stronger concept, say \emph{criticality} of Schr\"{o}dinger operator, to characterize the \emph{sharpness/optimality} of an inequality than the "sharp/optimal constant". We recall the original definition in \cite{DFP}.
\begin{definition}
Assume that $H\geq0$. The operator $H$ is said to be \emph{subcritical} in $\mathbb{R}^{N}$ if there exists a nonzero nonnegative continuous function $W$ such that
\begin{equation*}
\langle Hu,u\rangle\geq \int_{\mathbb{R}^{N}}W|u|^{2}dx,\hspace{2mm}u\in C_{0}^{\infty}(\mathbb{R}^{N}).
\end{equation*}
Otherwise, we say that $H$ is \emph{critical} in $\mathbb{R}^{N}$.
\end{definition}
Devyver-Fraas-Pinchover proved in \cite{DFP} that the Schr\"{o}dinger operators $-\Delta-\frac{(N-2)^{2}}{n^{2}}W_{2}$ and $-\Delta-\left(\frac{N-2}{n+1}\right)^{2}\left(W_{1}+W_{2}\right)$ are both critical, this implies that one cannot add any positive reminders to the right hand side of inequalities (\ref{T-1-5}) and (\ref{T-1-6}). With this in mind, one may want to know whether the following inequality (\ref{T-1-7}) is critical, since these three inequalities  are intrinsically related. In fact, assume $2-N\leq s\leq \frac{2-N}{2}$, we have (from Lemma \ref{lemma2.2} and identity (\ref{T-2-1})), for any $u\in C_{0}^{\infty}(\mathbb{R}^{N})$, the following inequality
\begin{equation}\label{T-1-7}
\begin{split}
\int_{\mathbb{R}^{N}}|\nabla u|^{2}dx\geq & [s(N-2)-s^{2}]\sum_{i=1}^{n}\int_{\mathbb{R}^{N}}\alpha_{i}\frac{|u|^{2}}{|x-a_{i}|^2}dx \\
     & +s^{2}\sum_{1\leq i<j\leq n}^{n}\int_{\mathbb{R}^{N}}\alpha_{i}\alpha_{j}\frac{\left|a_{i}-a_{j}\right|^{2}}{|x-a_{i}|^{2}|x-a_{j}|^{2}}|u|^{2}dx
\end{split}
\end{equation}
holds, where $\alpha_{i}\geq 0$, $i=1,\ldots,n$, $\sum_{i=1}^{n}\alpha_{i}=1$. One can also prove that, by the method introduced in \cite{DFP} of  constructing the "optimal Hardy weight", the Schr\"{o}dinger operator
\begin{equation*}
-\Delta-[s(N-2)-s^{2}]\sum_{i=1}^{n}\frac{\alpha_{i}}{|x-a_{i}|^2}dx
 -s^{2}\sum_{1\leq i<j\leq n}^{n}\frac{\alpha_{i}\alpha_{j}\left|a_{i}-a_{j}\right|^{2}}{|x-a_{i}|^{2}|x-a_{j}|^{2}}
\end{equation*}
is also critical when $0<\alpha_{i}\leq\frac{1}{2}$, $i=1,\ldots,n$. Let $\alpha_{i}=\frac{1}{n}$ for $i=1,\ldots,n$, the inequalities (\ref{T-1-4}), (\ref{T-1-5}) and (\ref{T-1-6}) are all special cases of inequality (\ref{T-1-7}) if we let $s=\frac{2-N}{2}$, $2-N$ and $\frac{n(2-N)}{n+1}$ respectively.

In this paper we always assume $n\geq2$. We will work in the function space $D^{2,2}(\mathbb{R}^{N})$. The first goal of the present article is to obtain the optimal $L^{2}$-version multipolar Rellich inequality. To our shallow knowledge, there has been no result on multipolar Rellich inequality in the existing literatures. A universal approach to obtain Hardy and Rellich inequalities is the so-called method of super-solutions. However, it is not an easy work to compute the fourth-order derivative of a ground state function in the case of Rellich inequality. In this paper, we take an subtle iterative method to solve this  difficulty. The main result of the paper is as follows.
\begin{theorem}\label{thm1.3}
Assume $N\geq 5$, $a_{1},\cdots,a_{n}\in\mathbb{R}^{N}$. Then the following inequality holds for any $u\in C_{0}^{\infty}(\mathbb{R}^{N})$,
\begin{equation}\label{T-1-8}
\int_{\mathbb{R}^{N}}|\Delta u|^{2}dx \geq \frac{N^{2}(N-4)^{2}}{n^{4}}\int_{\mathbb{R}^{N}}V_{n}(x)|u|^{2}dx,
\end{equation}
where
\begin{equation*}
V_{n}(x):=\sum_{1\leq i<j\leq n}\frac{\left|a_{i}-a_{j}\right|^{2}}{|x-a_{i}|^{2}|x-a_{j}|^{2}}
     \left(\sum_{1\leq k<l\leq n}\frac{\nu_{k,i,j}\nu_{l,i,j}\left|a_{k}-a_{l}\right|^{2}}{|x-a_{k}|^{2}|x-a_{l}|^{2}}\right),
\end{equation*}
with
\begin{equation*}
\nu_{k,i,j}=\left\{\begin{array}{ll}
\frac{N+2n-4}{N} &, k=i\hspace{2mm}or\hspace{2mm} k=j;\\
\frac{N-4}{N}    &, otherwise.\\
\end{array}\right.
\end{equation*}
The constant $\frac{N^{2}(N-4)^{2}}{n^{4}}$ is sharp. Furthermore, the constant $\frac{N^{2}(N-4)^{2}}{n^{4}}$ is achieved in the space $D^{2,2}(\mathbb{R}^{N})$ when $n\geq3$ by the minimizers (up to a constant)
\begin{equation*}
\phi(x)=\prod\limits_{i=1}^{n}|x-a_{i}|^{\frac{4-N}{n}},
\end{equation*}
while it can not be attained in the case $n=2$.
\end{theorem}
Especially, $\nu_{k,i,j}=1$ ($\forall i,j,k=1,2$) when $n=2$, we thus have
\begin{corollary}\label{cor1.4}
Assume $N\geq 5$, $a_{1},\cdots,a_{n}\in\mathbb{R}^{N}$.  The following inequality holds for any $u\in C_{0}^{\infty}(\mathbb{R}^{N})$,
\begin{equation}\label{T-1-9}
\int_{\mathbb{R}^{N}}|\Delta u|^{2}dx \geq \frac{N^{2}(N-4)^{2}}{16}
\int_{\mathbb{R}^{N}}\frac{|a_{1}-a_{2}|^{4}}{|x-a_{1}|^{4}|x-a_{2}|^{4}}|u|^{2}dx,
\end{equation}
where the constant $\frac{N^{2}(N-4)^{2}}{16}$ is sharp and it cannot be attained in $D^{2,2}(\mathbb{R}^{N})$.
\end{corollary}
Note that it is a surprising fact that the sharp constant in (\ref{T-1-9}) is the same as the classical sharp Rellich constant with single pole.

Next, we consider the criticality of the biharmonic Schr\"{o}dinger operator
\begin{equation*}
H_{n}:=\Delta^{2}-\frac{N^{2}(N-4)^{2}}{n^{4}}V_{n}.
\end{equation*}
We now introduce some definitions for the classification of the critical Schr\"{o}dinger operator.
\begin{definition}
Assume the Schr\"{o}dinger/biharmonic Schr\"{o}dinger operator $H$ is \emph{critical} in $\mathbb{R}^{N}$.  $H$ is said to be \emph{positive-critical} in $\mathbb{R}^{N}$ if the sharp constant $\lambda^{*}$ in (\ref{T-1-3}) is achieved in $D^{\sigma,2}(\mathbb{R}^{N})$. Otherwise, $H$ is called \emph{null-critical} in $\mathbb{R}^{N}$.
\end{definition}

The attainability of constant $\frac{N^{2}(N-4)^{2}}{n^{4}}$ in (\ref{T-1-8}) implies, in fact, the criticality of $H_{n}$ when $n\geq3$. More precisely, we have the following criticality result of biharmonic Schr\"{o}dinger operator $H_{n}$ for $n\geq2$.
\begin{proposition}\label{prop1.6}
The biharmonic Schr\"{o}dinger operator $H_{n}=\Delta^{2}-\frac{N^{2}(N-4)^{2}}{n^{4}}V_{n}$ is critical in $\mathbb{R}^{N}$ for $n\geq2$. Moreover, $H_{n}$ is positive-critical for $n\geq3$ and null-critical for $n=2$.
\end{proposition}
The plan of this paper is as follows. In section 2 we use the method of super-solutions to obtain the $L^2$-case multipolar Rellich inequality (\ref{T-1-8}) with sharp constant. In section 3 we prove the attainability or non-attainability of the constant $\frac{N^{2}(N-4)^{2}}{n^{4}}$ and the criticality of the biharmonic Schr\"{o}dinger operator $H_{n}$.

\section{$L^2$-case multipolar Rellich inequality with sharp constant}
In this section we employ the method of super-solutions to prove the sharp multipolar Rellich inequality.  Recall that the following Hardy-type identity: Let positive function $\phi\in C^{2}(\Omega)$. It holds for any  $u\in C_{0}^{\infty}(\Omega)$ that
\begin{equation}\label{T-2-1}
\int_{\Omega}\left|\nabla u-\frac{\nabla \phi}{\phi}u\right|^{2}dx=\int_{\Omega}|\nabla u|^{2}dx+\int_{\Omega}\frac{\Delta \phi}{\phi}|u|^{2}dx.
\end{equation}
For nonnegative function $W(x)\in L_{loc}^{1}(\Omega)$, the method of super-solutions shows that if there exists a positive function $\phi$ which is a super-solution of the Schr\"{o}dinger operator $-\Delta-W$, i.e. $(-\Delta-W)\phi\geq 0$, then $-\Delta-W\geq0$, which means that
\begin{equation*}
\int_{\Omega}|\nabla u|^{2}dx\geq\int_{\Omega}W|u|^{2}dx,\hspace{2mm} \forall u\in C_{0}^{\infty}(\Omega).
\end{equation*}
Thus, one only need to choose suitable $\phi$ to obtain Hardy inequalities for any potentials $W\leq \frac{-\Delta \phi}{\phi}$. This method can yield most of the well-known Hardy inequalities, for example, choosing $\phi(x)=|x|^{\frac{2-N}{2}}$ to yield the classical Hardy inequality and $\phi_{s}(x)=\prod_{i=1}^{n}|x-a_{i}|^{s\alpha_{i}}$ to obtain inequality (\ref{T-1-7}) (see also Lemma \ref{lemma2.2}).

In order to  apply the method of super-solutions to obtain the multipolar Rellich inequality, we need to find an Rellich-type identity analogous to (\ref{T-2-1})  for $\Omega=\mathbb{R}^{N}$, and choose a special auxiliary function  $\phi$. Assume $0<\phi\in C^{4}(\Omega)$ for now, we start from the following identity
\begin{equation*}
\int_{\Omega}\left|\Delta u-\frac{\Delta \phi}{\phi}u\right|^{2}dx=\int_{\Omega}|\Delta u|^{2}dx+\int_{\Omega}\frac{(\Delta \phi)^{2}}{\phi^{2}}|u|^{2}dx
                   -\int_{\Omega}2u\Delta u\frac{\Delta \phi}{\phi}dx.
\end{equation*}
Denoting $v=u\phi^{-1}$, then,
\begin{equation*}
\begin{split}
\int_{\Omega}\left|\Delta u-\frac{\Delta \phi}{\phi}u\right|^{2}dx=  & \int_{\Omega}|\Delta u|^{2}dx+\int_{\Omega}(\Delta \phi)^{2}|v|^{2}dx
                   -\int_{\Omega}2v\Delta(v\phi)\Delta \phi dx. \\
  =  & \int_{\Omega}|\Delta u|^{2}dx-\int_{\Omega}(\Delta \phi)^{2}|v|^{2}dx
                   -\int_{\Omega}2v\Delta \phi(\phi\Delta v+2\nabla\phi\cdot\nabla v)dx \\
  =  & \int_{\Omega}|\Delta u|^{2}dx-\int_{\Omega}\Delta \phi\left(\Delta \phi|v|^{2}
                   +2v\Delta v\phi+2\nabla \phi\cdot\nabla(|v|^{2})\right)dx  \\
  =  & \int_{\Omega}|\Delta u|^{2}dx-\int_{\Omega}\Delta \phi\left(\Delta(\phi|v|^{2})
                   -2|\nabla v|^{2}\phi\right)dx \\
  =  & \int_{\Omega}|\Delta u|^{2}dx-\int_{\Omega}\frac{\Delta^{2} \phi}{\phi}|u|^{2}dx
                   +2\int_{\Omega}\phi\Delta\phi|\nabla v|^{2}dx \\
\end{split}
\end{equation*}
Thus, we obtain
\begin{equation}\label{T-2-2}
 \int_{\Omega}|\Delta u|^{2}dx-\int_{\Omega}\frac{\Delta^{2} \phi}{\phi}|u|^{2}dx=\int_{\Omega}|\phi\Delta v+2\nabla\phi\cdot\nabla v|^{2}dx
                   -2\int_{\Omega}\phi\Delta\phi|\nabla v|^{2}dx.
\end{equation}
\begin{lemma}\label{lemma2.1}
\emph{(See also \cite{C})} Assume $W(x)\in L_{loc}^{1}(\Omega)$ is a nonnegative function. If $0<\phi\in C^{4}(\Omega)$ and $-\Delta\phi\geq0$ in $\Omega$, such that the following partial differential inequality holds
\begin{equation*}
(\Delta^{2}-W)\phi\geq 0\hspace{2mm}in\hspace{2mm}\Omega,
\end{equation*}
then, we have the following Rellich inequality
\begin{equation*}
\int_{\Omega}|\Delta u|^{2}dx\geq\int_{\Omega}W|u|^{2}dx,\hspace{2mm} \forall u\in C_{0}^{\infty}(\Omega).
\end{equation*}
\end{lemma}

We deduce from Lemma \ref{lemma2.1} that: if there is a positive function $\phi$ with  $\frac{\Delta^{2} \phi}{\phi}\in L_{loc}^{1}(\mathbb{R}^{N})$, then it holds
\begin{equation*}
\int_{\mathbb{R}^{N}}|\Delta u|^{2}dx\geq\int_{\mathbb{R}^{N}}\frac{\Delta^{2} \phi}{\phi}|u|^{2}dx,\hspace{2mm}\forall u\in C_{0}^{\infty}(\mathbb{R}^{N}).
\end{equation*}
 We now establish the following  lemma.
\begin{lemma}\label{lemma2.2}
Let $\phi_{s}(x)=\prod\limits_{i=1}^{n}|x-a_{i}|^{s\alpha_{i}}$, $s\in \mathbb{R}$, $0\leq\alpha_{i}\leq1$ and $\sum_{k=1}^{n}\alpha_{i}=1$.  Then
\begin{equation}\label{T-2-3}
\Delta\phi_{s}=\phi_{s}\left(s(N+s-2)\sum_{i=1}^{n}\frac{\alpha_{i}}{|x-a_{i}|^2}
         -s^{2}\sum_{1\leq i<j\leq n}^{n}\frac{\alpha_{i}\alpha_{j}\left|a_{i}-a_{j}\right|^{2}}{|x-a_{i}|^{2}|x-a_{j}|^{2}}\right).
\end{equation}
\end{lemma}
\begin{proof}
Rewriting $\phi_{s}$ with the form
\begin{equation*}
\phi_{s}(x)=\prod\limits_{i=1}^{n}\phi_{i}^{\alpha_{i}},
\end{equation*}
where $\phi_{i}=|x-a_{i}|^{s}$. By direct computation,
\begin{equation*}
\Delta \phi_{s}=\phi_{s}\left|\sum_{i=1}^{n}\alpha_{i}\frac{\nabla \phi_{i}}{\phi_{i}}\right|^{2}+\phi_{s}\left(\sum_{i=1}^{n}\alpha_{i}\frac{\Delta \phi_{i}}{\phi_{i}}-\sum_{i=1}^{n}\alpha_{i}\frac{|\nabla \phi_{i}|^{2}}{\phi_{i}^{2}}\right).
\end{equation*}
Simplifying the expression we obtain
\begin{equation}\label{T-2-4}
\Delta\phi_{s}=\phi_{s}\left(\sum_{i=1}^{n}\alpha_{i}\frac{\Delta \phi_{i}}{\phi_{i}}-\sum_{1\leq i<j\leq n}\alpha_{i}\alpha_{j}
          \left|\frac{\nabla \phi_{i}}{\phi_{i}}-\frac{\nabla \phi_{j}}{\phi_{j}}\right|^{2}\right).
\end{equation}
Since $\phi_{i}=|x-a_{i}|^{s}$,
\begin{equation}\label{T-2-5}
\frac{\Delta \phi_{i}}{\phi_{i}}=s(N+s-2)\frac{1}{|x-a_{i}|^{2}},\hspace{4mm}
\frac{|\nabla \phi_{i}|^{2}}{\phi_{i}^{2}}=s^{2}\frac{1}{|x-a_{i}|^{2}}.
\end{equation}
Inserting  (\ref{T-2-5}) into (\ref{T-2-4}) we have
\begin{equation*}
\Delta\phi_{s}=\phi_{s}\left(s(N+s-2)\sum_{i=1}^{n}\alpha_{i}\frac{1}{|x-a_{i}|^{2}} -s^{2}\sum_{i<j}^{n}\alpha_{i}\alpha_{j}\frac{\left|a_{i}-a_{j}\right|^{2}}{|x-a_{i}|^{2}|x-a_{j}|^{2}}\right).
\end{equation*}
Thus the proof of Lemma \ref{lemma2.2} is complete.
\end{proof}
In order to compute the $\Delta^{2}\phi_{s}$, we  write $\Delta\phi_{s}$ as
\begin{equation}\label{T-2-6}
\Delta \phi_{s}=s(N+s-2)\sum_{i=1}^{n}\alpha_{i}\hat{\phi}_{s,i}
         -s^{2}\sum_{1\leq i<j\leq n}^{n}\alpha_{i}\alpha_{j}\left|a_{i}-a_{j}\right|^{2}\hat{\phi}_{s,i,j}.
\end{equation}
where $\hat{\phi}_{s,i}=\prod\limits_{k=1}^{n}|x-a_{k}|^{(s-2)\xi_{k,i}}$, with
\begin{equation*}
\xi_{k,i}=\left\{\begin{array}{ll} \frac{s\alpha_{k}-2}{s-2}\hspace{5mm}&, k=i;\\
                                   \frac{s\alpha_{k}}{s-2}\hspace{5mm}&, k\neq i,\\
   \end{array}\right.
\end{equation*}
and $\hat{\phi}_{s,i,j}=\prod\limits_{k=1}^{n}|x-a_{k}|^{(s-4)\zeta_{k,i,j}}$, with
\begin{equation*}
\zeta_{k,i,j}=\left\{\begin{array}{ll} \frac{s\alpha_{k}-2}{s-4}\hspace{5mm}&, k=i\hspace{2mm}  or \hspace{2mm} k=j;\\
                                       \frac{s\alpha_{k}}{s-4}\hspace{5mm}&, otherwise.\\
   \end{array}\right.
\end{equation*}
It is easy to check that
\begin{equation*}
\sum_{k=1}^{n}\xi_{k,i}=1, \hspace{2mm} \hspace{2mm}\sum_{k=1}^{n}\zeta_{k,i,j}=1.
\end{equation*}
Applying Lemma \ref{lemma2.2} to $\hat{\phi}_{i}$ and $\hat{\phi}_{i,j}$ we have
\begin{equation}\label{T-2-7}
\Delta \hat{\phi}_{s,i}=\hat{\phi}_{s,i}\left((s-2)(N+s-4)\sum_{k=1}^{n}\frac{\xi_{k,i}}{|x-a_{k}|^2}
         -(s-2)^{2}\sum_{1\leq k<l\leq n}^{n}\frac{\xi_{k,i}\xi_{l,i}\left|a_{k}-a_{l}\right|^{2}}{|x-a_{k}|^{2}|x-a_{l}|^{2}}\right),
\end{equation}
\begin{equation}\label{T-2-8}
\Delta \hat{\phi}_{s,i,j}=\hat{\phi}_{s,i,j}\left((s-4)(N+s-6)\sum_{k=1}^{n}\frac{\zeta_{k,i,j}}{|x-a_{k}|^2}
         -(s-4)^{2}\sum_{1\leq k<l\leq n}^{n}\frac{\zeta_{k,i,j}\zeta_{l,i,j}\left|a_{k}-a_{l}\right|^{2}}{|x-a_{k}|^{2}|x-a_{l}|^{2}}\right).
\end{equation}
Since $\Delta$ is a linear operator, then
\begin{equation}\label{T-2-9}
\Delta^{2}\phi_{s}=s(N+s-2)\sum_{i=1}^{n}\alpha_{i}\Delta\hat{\phi}_{s,i}
         -s^{2}\sum_{1\leq i<j\leq n}^{n}\alpha_{i}\alpha_{j}\left|a_{i}-a_{j}\right|^{2}\Delta\hat{\phi}_{s,i,j}.
\end{equation}
Taking (\ref{T-2-7}) and (\ref{T-2-8}) into (\ref{T-2-9}), we get
\begin{equation*}
\begin{split}
\frac{\Delta^{2}\phi_{s}}{\phi_{s}}= & s(N+s-2)(s-2)(N+s-4)\sum_{i=1}^{n}\frac{\alpha_{i}}{|x-a_{i}|^2}
                          \left(\sum_{k=1}^{n}\frac{\xi_{k,i}}{|x-a_{k}|^2}\right) \\
    & -s(N+s-2)(s-2)^{2}\sum_{i=1}^{n}\frac{\alpha_{i}}{|x-a_{i}|^2}
                          \left(\sum_{1\leq k<l\leq n}^{n}\frac{\xi_{k,i}\xi_{l,i}\left|a_{k}-a_{l}\right|^{2}}{|x-a_{k}|^{2}|x-a_{l}|^{2}}\right) \\
    & -s^{2}(s-4)(N+s-6)\sum_{1\leq i<j\leq n}^{n}\frac{\alpha_{i}\alpha_{j}\left|a_{i}-a_{j}\right|^{2}}{|x-a_{i}|^{2}|x-a_{j}|^{2}}
                          \left(\sum_{k=1}^{n}\frac{\zeta_{k,i,j}}{|x-a_{k}|^2}\right) \\
    & +s^{2}(s-4)^{2}\sum_{1\leq i<j\leq n}^{n}\frac{\alpha_{i}\alpha_{j}\left|a_{i}-a_{j}\right|^{2}}{|x-a_{i}|^{2}|x-a_{j}|^{2}}
                          \left(\sum_{1\leq k<l\leq n}^{n}\frac{\zeta_{k,i,j}\zeta_{k,i,j}\left|a_{k}-a_{l}\right|^{2}}{|x-a_{k}|^{2}|x-a_{l}|^{2}}\right).
\end{split}
\end{equation*}
Thus, we arrive at
\begin{equation}\label{T-2-10}
\begin{split}
\frac{\Delta^{2}\phi_{s}}{\phi_{s}}= & s(N+s-2)(s-2)(N+s-4)\sum_{i=1}^{n}\frac{\alpha_{i}}{|x-a_{i}|^2}
                          \left(\sum_{k=1}^{n}\frac{\xi_{k,i}}{|x-a_{k}|^2}\right) \\
    & +2s^{2}(s-4)(4-N-s)\sum_{1\leq i<j\leq n}^{n}\frac{\alpha_{i}\alpha_{j}\left|a_{i}-a_{j}\right|^{2}}{|x-a_{i}|^{2}|x-a_{j}|^{2}}
                          \left(\sum_{k=1}^{n}\frac{\zeta_{k,i,j}}{|x-a_{k}|^2}\right) \\
    & +s^{2}(s-4)^{2}\sum_{1\leq i<j\leq n}^{n}\frac{\alpha_{i}\alpha_{j}\left|a_{i}-a_{j}\right|^{2}}{|x-a_{i}|^{2}|x-a_{j}|^{2}}
                          \left(\sum_{1\leq k<l\leq n}^{n}\frac{\zeta_{k,i,j}\zeta_{k,i,j}\left|a_{k}-a_{l}\right|^{2}}{|x-a_{k}|^{2}|x-a_{l}|^{2}}\right).
\end{split}
\end{equation}
The last equality holds because of the the following identity
\begin{equation*}
\sum_{i=1}^{n}\frac{\alpha_{i}}{|x-a_{i}|^2}
\left(\sum_{1\leq k<l\leq n}^{n}\frac{\xi_{k,i}\xi_{l,i}\left|a_{k}-a_{l}\right|^{2}}{|x-a_{k}|^{2}|x-a_{l}|^{2}}\right)
=\frac{s(s-4)}{(s-2)^{2}}\sum_{1\leq i<j\leq n}^{n}\frac{\alpha_{i}\alpha_{j}\left|a_{i}-a_{j}\right|^{2}}{|x-a_{i}|^{2}|x-a_{j}|^{2}}
                          \left(\sum_{k=1}^{n}\frac{\zeta_{k,i,j}}{|x-a_{k}|^2}\right).
\end{equation*}

\begin{theorem}\label{thm2.3}
Assume $4-N\leq s<0$, $N\geq 5$, $a_{1},\cdots,a_{n}\in\mathbb{R}^{N}$. Then  the following inequality holds for any $u\in C_{0}^{\infty}(\mathbb{R}^{N})$,
\begin{equation}\label{T-2-11}
\begin{split}
\int_{\mathbb{R}^{N}}&|\Delta u|^{2}dx \geq \frac{s(s-2)(2-N-s)(4-N-s)}{n^{2}}
\int_{\mathbb{R}^{N}}\sum_{1\leq i,k\leq n}\frac{\mu_{k,i}}{|x-a_{i}|^2|x-a_{k}|^{2}}|u|^{2}dx \\
     & +\frac{2s^{3}(4-N-s)}{n^{3}}\int_{\mathbb{R}^{N}}\sum_{1\leq i<j\leq n}\frac{\left|a_{i}-a_{j}\right|^{2}}{|x-a_{i}|^{2}|x-a_{j}|^{2}}
     \left(\sum_{1\leq k\leq n}\frac{\sigma_{k,i,j}}{|x-a_{k}|^{2}}\right)|u|^{2}dx \\
     & +\frac{s^{2}(s-4)^{2}}{n^{4}}\int_{\mathbb{R}^{N}}\sum_{1\leq i<j\leq n}\frac{\left|a_{i}-a_{j}\right|^{2}}{|x-a_{i}|^{2}|x-a_{j}|^{2}}
     \left(\sum_{1\leq k<l\leq n}\frac{\nu_{k,i,j}\nu_{l,i,j}\left|a_{k}-a_{l}\right|^{2}}{|x-a_{k}|^{2}|x-a_{l}|^{2}}\right)|u|^{2}dx,
\end{split}
\end{equation}
where
\begin{equation*}
\mu_{k,i}=\left\{\begin{array}{ll}
 \frac{s-2n}{s-2} &, k=i;\\
 \frac{s}{s-2}    &, k\neq i,\\
 \end{array}\right.\hspace{5mm}
\sigma_{k,i,j}=\left\{\begin{array}{ll}
\frac{s-2n}{s} &, k=i\hspace{2mm}  or \hspace{2mm} k=j;\\
1                      &,otherwise,\\
\end{array}\right.\hspace{5mm}
\nu_{k,i,j}=\left\{\begin{array}{ll}
\frac{s-2n}{s-4} &, k=i\hspace{2mm}  or \hspace{2mm} k=j;\\
\frac{s}{s-4}    &, otherwise.\\
\end{array}\right.
\end{equation*}
\end{theorem}
\begin{proof}
Taking $\alpha_{i}=\frac{1}{n}$, $i=1,\cdots,n$, since $N\geq 5$, $\phi_{s}>0$, and
\begin{equation}\label{T-2-12}
-\Delta\phi_{s}=\frac{s(2-N-s)}{n}\phi\sum_{i=1}^{n}\frac{1}{|x-a_{i}|^{2}}
         +\frac{s^{2}}{n^{2}}\phi\sum_{1\leq i<j\leq n}^{n}\frac{\left|a_{i}-a_{j}\right|^{2}}{|x-a_{i}|^{2}|x-a_{j}|^{2}}\geq0,
\end{equation}
Theorem \ref{thm2.3} follows from Lemma \ref{lemma2.1}.
\end{proof}

\emph{Proof of inequality (\ref{T-1-8})}: Taking $s=4-N$ in (\ref{T-2-11}) we obtain inequality (\ref{T-1-8}).

\section{Attainability and Criticality}
In this section we prove the attainability  of constant $\frac{N^{2}(N-4)^{2}}{n^{4}}$ in inequality (\ref{T-1-8}) and the criticality of the biharmonic schr\"{o}dinger operator $H_{n}$ for  $N\geq5$.

\subsection{Attainability for constant $\frac{N^{2}(N-4)^{2}}{n^{4}}$}
Let $\phi(x)=\prod\limits_{i=1}^{n}|x-a_{i}|^{\frac{4-N}{n}}$. In view of (\ref{T-2-10}), $\phi(x)$ is the minimizer or virtual minimizer in inequality (\ref{T-1-8}). We need to check the integrability of $|\Delta\phi|^{2}$.
Let $s=4-N$ in (\ref{T-2-12}) we have
\begin{equation}\label{T-3-1}
\Delta\phi=-2(N-4)\phi\sum_{i=1}^{n}\frac{\alpha_{i}}{|x-a_{i}|^{2}}
         -(N-4)^{2}\phi\sum_{1\leq i<j\leq n}^{n}\frac{\alpha_{i}\alpha_{j}\left|a_{i}-a_{j}\right|^{2}}{|x-a_{i}|^{2}|x-a_{j}|^{2}}.
\end{equation}
Denote
\begin{equation*}
\begin{split}
a_{i}(x)=&\frac{\phi(x)^{2}}{|x-a_{i}|^{4}},\hspace{2mm}b_{i,j}(x)=\frac{\phi(x)^{2}}{|x-a_{i}|^{2}|x-a_{j}|^{2}},
\hspace{2mm} c_{i,j}(x)=\frac{\phi(x)^{2}}{|x-a_{i}|^{4}|x-a_{j}|^{2}}, \\
d_{i,j,k}&(x)=\frac{\phi(x)^{2}}{|x-a_{i}|^{2}|x-a_{j}|^{2}|x-a_{k}|^{2}}, \hspace{2mm} e_{i,j}(x)=\frac{\phi(x)^{2}}{|x-a_{i}|^{4}|x-a_{j}|^{4}}, \\
f_{i,j,k,l}(x)=&\frac{\phi(x)^{2}}{|x-a_{i}|^{2}|x-a_{j}|^{2}|x-a_{k}|^{2}|x-a_{l}|^{2}},\hspace{2mm}
g_{i,j,k}(x)=\frac{\phi(x)^{2}}{|x-a_{i}|^{4}|x-a_{j}|^{2}|x-a_{k}|^{2}}.
\end{split}
\end{equation*}
Expanding $|\Delta\phi|^{2}$, we know that $|\Delta\phi|^{2}$ is just a linear combination of these seven classes of functions. The integrability results of these function is as follows.
\begin{proposition}\label{prop2.1}
$a_{i}(x)$, $b_{i,j}(x)$ $c_{i,j}(x)$, $d_{i,j,k}(x)$, $e_{i,j}(x)$, $f_{i,j,k,l}(x)$, and $g_{i,j,k}(x)$ are all in $L^{1}(\mathbb{R}^{N})$ when $n\geq3$, while $a_{i}(x)$, $c_{i,j}(x)$, $e_{i,j}(x)$ and $g_{i,j,k}(x)$ are not in $L^{1}(\mathbb{R}^{N})$ when $n=2$.
\end{proposition}
\begin{proof}
We only prove the integrability results for $a_{i}(x)$ and $e_{i,j}(x)$, the argument of others is similar. Assume $\rho>0$ small enough such that $\{a_{1},\cdots,a_{n}\}\subset B_{\frac{1}{\rho}}(0)$ and $B_{\rho}(a_{i})\bigcap B_{\rho}(a_{j})=\emptyset$ for any $i\neq j$. Since $a_{i}(x)$ and $e_{i,j}(x)$ are bounded functions in the bounded domain $B_{\frac{1}{\rho}}(0)\setminus\bigcup_{i=1}^{n}B_{\rho}(a_{i})$, we only need to consider the integrability of them in $B_{\frac{1}{\rho}}(0)^{c}$ and $B_{\rho}(a_{i})$ for $i=1,\cdots,n$. Note that $|x-a_{j}|\sim |a_{i}-a_{j}|$ for $x\in B_{\rho}(a_{i})$ when $i\neq j$ and $|x-a_{i}|\sim |x|$ for $x\in B_{\frac{1}{\rho}}(0)^{c}$, the following asymptotical behavior holds:
\begin{enumerate}
\item for $x\in B_{\rho}(a_{i})$,
\begin{equation*}
a_{i}(x)\simeq e_{i,j}(x)\simeq |x-a_{i}|^{\frac{(n-2)(N-4)}{n}-N};
\end{equation*}
\item for $x\in B_{\rho}(a_{j})$,
\begin{equation*}
|x-a_{j}|^{\frac{(n-2)(N-4)}{n}-N+4}\simeq a_{i}(x)\lesssim e_{i,j}(x)\simeq|x-a_{j}|^{\frac{(n-2)(N-4)}{n}-N};
\end{equation*}
\item for $x\in B_{\rho}(a_{k})$, $k\neq i,j$,
\begin{equation*}
a_{i}(x)\simeq e_{i,j}(x)\simeq |x-a_{i}|^{\frac{(n-2)(N-4)}{n}-N+4};
\end{equation*}
\item for $x\in B_{\frac{1}{\rho}}(0)^{c}$,
\begin{equation*}
|x|^{4-2N}\simeq a_{i}(x)\gtrsim e_{i,j}(x)\simeq|x|^{-2N};
\end{equation*}
\end{enumerate}
Therefore, it shows that $a_{i}(x)$, $e_{i,j}(x)\in L^{1}(B_{\frac{1}{\rho}}(0)^{c}\bigcup\cup_{i=1}^{n}B_{\rho}(a_{i}))$  when $n\geq 3$ and $N\geq 5$, whereas $a_{i}(x)$, $e_{i,j}(x)\notin L^{1}(B_{\rho}(a_{i}))$ when $n=2$. Thus, we prove the claim for $a_{i}(x)$, $e_{i,j}(x)$ in Proposition \ref{prop2.1} and the proof for other functions is similar. This complete the proof.
\end{proof}
\emph{Proof of the attainability:}
By (\ref{T-3-1}) and Proposition \ref{prop2.1}, we get the integrability of $|\Delta\phi|^{2}$ when $n\geq 3$, which means $\phi\in D^{2,2}(\mathbb{R}^{N})$. So, by using integration by parts, we get the following identity
\begin{equation}\label{}
\int_{\mathbb{R}^{N}}|\Delta\phi|^{2}dx=\int_{\mathbb{R}^{N}}\frac{\Delta^{2}\phi}{\phi}\phi^{2}dx
=\frac{N^{2}(N-4)^{2}}{n^{4}}\int_{\mathbb{R}^{N}}V_{n}(x)\phi^{2}dx.
\end{equation}
This implies the attainability of $\frac{N^{2}(N-4)^{2}}{n^{4}}$ in $D^{2,2}(\mathbb{R}^{N})$ when $n\geq 3$. When $n=2$, $\phi\notin D^{2,2}(\mathbb{R}^{N})$ since $a_{i}(x)$ is not in $L^{1}(\mathbb{R}^{N})$, which shows that the constant $\frac{N^{2}(N-4)^{2}}{16}$ is not attained in $D^{2,2}(\mathbb{R}^{N})$.

$\hfill\square$

\subsection{Criticality for $n=2$}
We know from Theorem \ref{thm1.3} that $H_n$ is is positive-critical for $n\geq3$, and we study the criticality for $n=2$ in this subsection. For general subcritical second-order elliptic operator, a convenient method to construct a critical schr\"{o}dinger operator is given in \cite{DFP} by constructing the \emph{optimal Hardy-weight}, but we have no such direct approach for  biharmonic schr\"{o}dinger operator. Therefore, we need to find other method to deal with the criticality of nonnegative operators.
\begin{lemma}
 A nonnegative operator $H$ defined in Hilbert space $L^{2}(\mathbb{R}^{N})$ is critical if there exists a function sequence $\{u_{\epsilon}\}_{\epsilon>0}$ and positive a.e. continuous function $\psi(x)$, such that
\begin{equation*}
\lim_{\epsilon\rightarrow 0}\langle Hu_{\epsilon},u_{\epsilon}\rangle=0,\hspace{2mm}and\hspace{2mm}\lim_{\epsilon\rightarrow 0}u_{\epsilon}=\psi(x)\hspace{2mm}a.e.\hspace{2mm}in\hspace{2mm}\mathbb{R}^{N}.
\end{equation*}
\end{lemma}
\begin{proof}
We assume that there exists a nonzero nonnegative potential $W$ such that
\begin{equation*}
\langle Hu,u\rangle\geq \int_{\mathbb{R}^{N}}W|u|^{2}dx.
\end{equation*}
Then, from Fatou lemma we have
\begin{equation*}
0\leq\int_{\mathbb{R}^{N}}W|\psi|^{2}dx\leq \liminf_{\epsilon\rightarrow 0}\int_{\mathbb{R}^{N}}W|u_{\epsilon}|^{2}dx\leq\lim_{\epsilon\rightarrow 0}\langle Hu_{\epsilon},u_{\epsilon}\rangle=0,
\end{equation*}
which froces that $W=0$ a.e. in $\mathbb{R}^{N}$. So, the proof is completed by contradiction.
\end{proof}

In view of the Rellich-type identity (\ref{T-2-2}), we have the following result to determine the criticality of biharmonic schr\"{o}dinger operator $H_{2}$.
\begin{corollary}\label{cor3.3}
Let $\phi$ satisfies the assumptions in Lemma \ref{lemma2.1}. The operator $\Delta^{2}-\frac{\Delta^{2}\phi}{\phi}$ is critical if there exists a function sequence $\{v_{\epsilon}\}_{\epsilon>0}$, with $0<v_{\epsilon}\leq 1$, such that
\begin{equation*}
\lim_{\epsilon\rightarrow 0}\left(\int_{\mathbb{R}^{N}}|\phi\Delta v_{\epsilon}+2\nabla\phi\cdot\nabla v_{\epsilon}|^{2}dx-\int_{\mathbb{R}^{N}}\phi\Delta\phi|\nabla v_{\epsilon}|^{2}dx\right)=0,\hspace{2mm}and\hspace{2mm}\lim_{\epsilon\rightarrow 0}v_{\epsilon}=1.
\end{equation*}
\end{corollary}

\emph{Proof of Proposition \ref{prop1.6} for $n=2$:} In this case $\phi(x)=|x-a_{1}|^{\frac{4-N}{2}}|x-a_{2}|^{\frac{4-N}{2}}$. Let's select a cut-off function as follows
\begin{equation*}
v_{\epsilon}(x)=\left\{\begin{array}{ll}
0                                                       &, x\in B_{\epsilon^{2}}(a_{i}),i=1,2;\\
\frac{\log{(|x-a_{i}|/\epsilon^{2})}}{\log{1/\epsilon}} &, x\in B_{\epsilon}(a_{i})\setminus B_{\epsilon^{2}}(a_{i}),i=1,2;\\
1                                                       &, x\in B_{\frac{1}{\epsilon}}(0)\setminus\bigcup_{i=1}^{n}B_{\epsilon}(a_{i}),i=1,2;\\
\frac{\log(\epsilon^{2}|x|)}{\log{\epsilon}}            &, x\in B_{\frac{1}{\epsilon^{2}}}(0)\setminus B_{\frac{1}{\epsilon}}(0),i=1,2;\\
0                                                       &, |x|\geq\frac{1}{\epsilon^{2}}.
   \end{array}\right.
\end{equation*}
By direct computation, the gradient and Laplacian of $v_{\epsilon}$ are respectivly
\begin{equation*}
\nabla v_{\epsilon}(x)=\left\{\begin{array}{ll}
\frac{x-a_{i}}{|x-a_{i}|^{2}}\left(\log{\frac{1}{\epsilon}}\right)^{-1} &, x\in B_{\epsilon}(a_{i})\setminus B_{\epsilon^{2}}(a_{i}),i=1,2;\\
\frac{x}{|x|^{2}} \log{\frac{1}{\epsilon}}                             &, x\in B_{\frac{1}{\epsilon^{2}}}(0)\setminus B_{\frac{1}{\epsilon}}(0),i=1,2;\\
0                                                                      &, otherwise.
   \end{array}\right.
\end{equation*}
and
\begin{equation*}
\Delta v_{\epsilon}(x)=\left\{\begin{array}{ll}
\frac{N-2}{|x-a_{i}|^{2}}\left(\log{\frac{1}{\epsilon}}\right)^{-1} &, x\in B_{\epsilon}(a_{i})\setminus B_{\epsilon^{2}}(a_{i}),i=1,2;\\
\frac{N-2}{|x|^{2}}\log{\frac{1}{\epsilon}}                         &, x\in B_{\frac{1}{\epsilon^{2}}}(0)\setminus B_{\frac{1}{\epsilon}}(0),i=1,2;\\
0                                                                   &, otherwise.
   \end{array}\right.
\end{equation*}
On the other hand, we have
\begin{equation*}
\begin{split}
         & \int_{\mathbb{R}^{N}}|\phi\Delta v_{\epsilon}+2\nabla\phi\cdot\nabla v_{\epsilon}|^{2}dx
           -\int_{\mathbb{R}^{N}}\phi\Delta\phi|\nabla v_{\epsilon}|^{2}dx \\
\lesssim & \int_{\mathbb{R}^{N}}|\phi\Delta v_{\epsilon}|^{2}dx
           +\int_{\mathbb{R}^{N}}|\nabla\phi\cdot\nabla v_{\epsilon}|^{2}dx
           +\int_{\mathbb{R}^{N}}|\phi\Delta\phi||\nabla v_{\epsilon}|^{2}dx \\
   =:    & I_{1}+I_{2}+I_{3}.
\end{split}
\end{equation*}
Using spherical coordinate formula, we get
\begin{equation*}
\begin{split}
I_{1}\simeq & \left(\log{\frac{1}{\epsilon}}\right)^{-2}|a_{1}-a_{2}|^{4-N}\sum_{i=1}^{2}\int_{B_{\epsilon}(a_{i})\setminus B_{\epsilon^{2}}(a_{i})}
              |x-a_{i}|^{4-N}\cdot\frac{(N-2)^{2}}{|x-a_{i}|^{4}}dx \\
            & +\left(\log{\frac{1}{\epsilon}}\right)^{2}\int_{B_{\frac{1}{\epsilon^{2}}}(0)\setminus B_{\frac{1}{\epsilon}}(0)}
              |x|^{2(4-N)}\cdot\frac{(N-2)^{2}}{|x|^{4}}dx \\
     \simeq & \left(\log{\frac{1}{\epsilon}}\right)^{-2}\int_{\epsilon^{2}}^{\epsilon}r^{-1}dr
              +\left(\log{\frac{1}{\epsilon}}\right)^{2}\int_{\frac{1}{\epsilon}}^{\frac{1}{\epsilon^{2}}}r^{4-N-1}dr \\
     \simeq & \left(\log{\frac{1}{\epsilon}}\right)^{-1}+\left(\frac{1}{\epsilon}\right)^{4-N}\left(\log{\frac{1}{\epsilon}}\right)^{2}
              \rightarrow 0\hspace{2mm}as\hspace{2mm}\epsilon\rightarrow 0.
\end{split}
\end{equation*}
By similar computation, we have
\begin{equation*}
\begin{split}
I_{2}\simeq & \left(\log{\frac{1}{\epsilon}}\right)^{-2}\sum_{i=1}^{2}\int_{B_{\epsilon}(a_{i})\setminus B_{\epsilon^{2}}(a_{i})}
              |x-a_{i}|^{-N}dx+\left(\log{\frac{1}{\epsilon}}\right)^{2}\int_{B_{\frac{1}{\epsilon^{2}}}(0)\setminus B_{\frac{1}{\epsilon}}(0)}|x|^{4-2N}dx \\
            & \rightarrow 0\hspace{2mm}as\hspace{2mm}\epsilon\rightarrow 0.
\end{split}
\end{equation*}
and
\begin{equation*}
\begin{split}
I_{3}\simeq & \left(\log{\frac{1}{\epsilon}}\right)^{-2}\sum_{i=1}^{2}\int_{B_{\epsilon}(a_{i})\setminus B_{\epsilon^{2}}(a_{i})}
              |x-a_{i}|^{-N}dx+\left(\log{\frac{1}{\epsilon}}\right)^{2}\int_{B_{\frac{1}{\epsilon^{2}}}(0)\setminus B_{\frac{1}{\epsilon}}(0)}|x|^{4-2N}dx \\
            & \rightarrow 0\hspace{2mm}as\hspace{2mm}\epsilon\rightarrow 0.
\end{split}
\end{equation*}
Summing over $I_{1}$, $I_{2}$ and $I_{3}$ yields
\begin{equation*}
\lim_{\epsilon\rightarrow 0}\left(\int_{\mathbb{R}^{N}}|\phi\Delta v_{\epsilon}+2\nabla\phi\cdot\nabla v_{\epsilon}|^{2}dx-\int_{\mathbb{R}^{N}}\phi\Delta\phi|\nabla v_{\epsilon}|^{2}dx\right)=0.
\end{equation*}
Thus the operator $H_{2}$ is critical by Corollary \ref{cor3.3}, moreover, it is null-critical by Corollary \ref{cor1.4}.

$\hfill\square$


\end{document}